\documentclass[psamsfonts]{amsart}

\usepackage[margin=2.5cm]{geometry}
\usepackage[all,arc]{xy}
\usepackage{amsfonts}
\usepackage{amssymb}
\usepackage{amsaddr}
\usepackage{enumerate}
\usepackage{eucal}
\usepackage{tikz}
\usepackage{parskip}
\usepackage{mathtools}

\newtheorem{thm}{Theorem}[section]
\newtheorem{cor}[thm]{Corollary}
\newtheorem{prop}[thm]{Proposition}
\newtheorem{lemma}[thm]{Lemma}

\theoremstyle{definition}
\newtheorem{defn}[thm]{Definition}

\theoremstyle{remark}
\newtheorem{rem}[thm]{Remark}

\makeatletter
\let\c@equation\c@thm
\makeatother
\numberwithin{equation}{section}

\DeclareMathOperator{\ad}{ad}

\DeclareMathOperator{\en}{End}

\DeclareMathOperator{\Hom}{Hom}

\DeclareMathOperator{\res}{Res}

\DeclareMathOperator{\str}{STr}

\DeclareMathOperator{\sdim}{sdim}

\newcommand{\even}{\overline{0}}
\newcommand{\odd}{\overline{1}}

\newcommand{\C}{\mathbb{C}}

\newcommand{\R}{\mathbb{R}}

\newcommand{\Z}{\mathbb{Z}}

\newcommand{\g}{\mathfrak{g}}
\newcommand{\h}{\mathfrak{h}}

\newcommand{\vac}{{\left|0\right>}}

\newcommand{\twobytwo}[4]
{\left(\begin{smallmatrix} #1 & #2 \\ #3 & #4 \end{smallmatrix}\right)}

\newcommand{\slfrak}{\mathfrak{sl}}
\newcommand{\glfrak}{\mathfrak{gl}}

\newcommand{\ospfrak}{\mathfrak{osp}}
\newcommand{\spfrak}{\mathfrak{sp}}
\newcommand{\sofrak}{\mathfrak{so}}

\begin{document}
\vspace{8mm}

\begin{center}

\textbf{\Large{Holomorphic integer graded vertex superalgebras}}

\vspace{6mm}
		
{\large Jethro van Ekeren\footnote{\texttt{jethrovanekeren@gmail.com}} and Bely Rodr\'{i}guez Morales\footnote{\texttt{belyto.rodriguez@gmail.com}}}

\vspace{3mm}

\normalsize
{\slshape
Instituto de Matem\'{a}tica e Estat\'{i}stica\\
Universidade Federal Fluminense\\
Niter\'{o}i, RJ 24.210-201, Brazil }

\end{center}

\vspace*{1mm}

\begin{abstract}\noindent 
In this note we study holomorphic $\Z$-graded vertex superalgebras. We prove that all such vertex superalgebras of central charge $8$ and $16$ are purely even. For the case of central charge $24$ we prove that the weight-one Lie superalgebra is either zero, of superdimension $24$, or else is one of an explicit list of 1332 semisimple Lie superalgebras.
\end{abstract}

\title{}

\maketitle

%
%
%
%
%
%
%
%
%

\section{Introduction}

We call a rational vertex algebra or superalgebra \emph{holomorphic} if it possesses a unique irreducible module up to isomorphism, and satisfies some additional technical conditions, detailed in Definition \ref{def:holom} below. The celebrated theorem of Zhu \cite{Zhu} on modularity of characters of vertex algebras imposes highly nontrivial constraints on the structure of holomorphic vertex algebras. For instance the central charge of such a vertex algebra must be a multiple of $8$.

In the remarkable work of Schellekens \cite{Sch93} an explicit, and conjecturally complete, list of 71 candidate holomorphic vertex algebras (there called meromorphic conformal field theories) of central charge $24$ was given, indexed by the Lie algebra structure of the weight-one space. Indeed let $V$ be such a vertex algebra, then the weight-one space $V_1$ is either $\{0\}$, abelian of dimension $24$, or else is semisimple and its simple components $\g^{(1)}$, $\g^{(2)}$... $\g^{(r)}$ satisfy
\begin{align}\label{eq:h.k.eq.intro}
\frac{h_i^\vee}{k_i} = \frac{\dim(V_1)-24}{24}, \quad \text{for $i = 1, 2, \ldots, r$}
\end{align}
(here $h_i^\vee$ is the dual Coxeter number of $\g^{(i)}$ and $k_i \in \Z_{\geq 1}$ is the level of the corresponding affine vertex subalgebra of $V$). This has since been proven as a theorem of vertex algebras \cite{DM04} \cite{DM06} (see also \cite{Tuite}).

The first step in Schellekens' classification consisted in identifying the solutions to \eqref{eq:h.k.eq.intro}, which turn out to be 221 in number. The reduction to the final list of 71 relied on the solution of complicated integer programming problems (see also \cite{EMS20}), though recently in \cite{ELMS21} a more conceptual approach to this reduction has been given, based on orbifolds of the Leech lattice vertex algebra and Kac's very strange formula.

Our purpose is to generalise \eqref{eq:h.k.eq.intro} above to the context of $\Z$-graded vertex superalgebras and to use it to put constraints on possible Lie superalgebra structures of their weight-one spaces in central charge $24$. The results are sumarised in Theorem \ref{thm:main.24} in which we determine all possible Lie superalgebra structures on $V_1$ outside the case $\sdim(V_1)=24$. We also investigate central charges $c=8$ and $c=16$, ruling out the existence of non purely even holomophic $\Z$-graded vertex superalgebras in these cases (Theorem \ref{thm:main.8.16}).

The most commonly studied class of vertex superalgebras are those which are $(1/2)\Z$-graded, or more precisely those for which even elements are required to have integer degree and odd elements half-odd-integer degree. The lattice vertex superalgebra associated with an odd lattice, with its standard conformal structure is of this type for example. 

The class of $\Z$-graded vertex superalgebras is much less studied, though there are plenty of examples: vertex superalgebras associated with untwisted affine Lie superalgebras, odd lattice vertex superalgebras with modified conformal structure (though these examples are typically not self-contragredient), certain $W$-superalgebras, among others.

The question of classification of holomorphic $(1/2)\Z$-graded vertex superalgebras has been considered (in \cite{CDR2017} for example), but is rather different from the $\Z$-graded case. The first main difference is that the weight-one space of a $(1/2)\Z$-graded vertex superalgebra is always a purely even Lie algebra, so the representation theory of Lie superalgebras does not enter. The second main difference is that the modular properties of characters in the $(1/2)\Z$-graded case are more complicated than in the $\Z$-graded case.

The results obtained in this note are partial, not least because the representation theory of Lie superalgebras is significantly harder than that of Lie algebras. We nevertheless hope that they serve as a starting point for further work on the structure of this understudied class of algebras.

\emph{Acknowledgements} The authors would like to thank L. Calixto, T. Creutzig and S. M\"{o}ller for comments on a draft of this work. JvE was supported by CNPq grant 310576/2020-2, FAPERJ grant 211.358/2019, and a grant from the Serrapilheira Institute (grant number Serra - 1912-31433). BRM was supported by a PNPD postdoc from CAPES.

\section{Lie superalgebras and supertrace forms}

We briefly recall some background material on finite dimensional simple Lie superalgebras from \cite{Kac77} and \cite{Kac77sketch}. A Lie superalgebra $\g$ is a vector superspace ($\Z/2$-graded vector space) $\g = \g_{\even} \oplus \g_{\odd}$ with bilinear bracket operation $[\cdot, \cdot] : \g \times \g \rightarrow \g$ satisfying
\[
[b, a] = -p(a, b) [a, b]
\]
and, writing $\ad(a)$ for the operation $x \mapsto [a, x]$,
\[
\ad(a)\ad(b) - p(a, b) \ad(b)\ad(a) = \ad([a, b]).
\]
Here and always, for homogeneous elements $a, b$ of degrees $p(a)$ and $p(b)$ of a vector superspace, $p(a, b)$ denotes $(-1)^{p(a)p(b)}$. If $U$ is a vector superspace then $\mathfrak{gl}(U) = \en(U)$ equipped with the induced grading and
\[
[a, b] = ab - p(a, b) ba
\]
is an example of a Lie superalgebra.

Homomorphisms are $\Z/2$-graded linear maps preserving brackets, and a $\g$-module is a vector superspace $M$ together with a homomorphism $\g \rightarrow \mathfrak{gl}(M)$. Similarly subalgebras and ideals are by definition $\Z/2$-graded. The linear dual of a $\g$-module $M$ is again a $\g$-module with action defined by
\[
[a \cdot f](m) = -p(a, f) f(a \cdot m).
\]
A bilinear form $(\cdot, \cdot) : U \times U \rightarrow \C$ on the vector superspace $U$ is supersymmetric if
\[
(v, u) = p(u, v) (u, v)
\]
for all homogeneous $u, v \in U$. A bilinear form $(\cdot, \cdot)$ on the Lie superalgebra $\g$ is said to be invariant if
\[
([a, x], y) = -p(a, x) (x, [a, y])
\]
for all $a, x, y \in \g$. 
Now suppose that $\g$ is finite dimensional. The Killing form $\kappa : \g \times \g \rightarrow \C$ of $\g$ is defined by
\[
\kappa(x, y) = \str_{\g} \ad(x) \ad(y).
\]
This form is supersymmetric and invariant, and is also ``consistent'', which means that $\kappa(\g_{\even}, \g_{\odd}) = 0$.

As in the theory of Lie algebras, a Lie superalgebra is said to be simple if it contains no nontrivial ideal, and semisimple if it contains no nontrivial solvable ideal, where the notion of solvable is defined as for ordinary Lie algebras.

An invariant bilinear form on a simple Lie superalgebra, if it exists, is unique up to scaling \cite[Proposition 1.2.6]{Kac77} and is either nondegenerate or zero. Unlike the case of ordinary Lie algebras, the Killing form of a simple Lie superalgebra need not be nondegenerate. This makes the general theory much less well behaved than the purely even case. Nevertheless one has:
\begin{prop}[{\cite[Proposition 2.3.3]{Kac77}}]\label{prop.nondeg=good}
Let $\g$ be a finite dimensional Lie superalgebra whose Killing form is nondegenerate. Then $\g$ is isomorphic to a direct sum of simple Lie superalgebras, each with nondegenerate Killing form.
\end{prop}

A simple Lie superalgebra $\g$ is said to be of \emph{classical type} if the induced representation of $\g_{\even}$ on $\g_{\odd}$ is completely reducible.
\begin{prop}[{\cite[Proposition 2.3.6]{Kac77}}]
Let $\g$ be a finite dimensional simple Lie superalgebra. If the Killing form of $\g$ is nondegenerate then $\g$ is of classical type.
\end{prop}

The classification of finite dimensional simple Lie superalgebras was achieved by Kac in \cite{Kac77}, and is as follows.
\begin{prop}
{\ }
The (non purely even) finite dimensional simple Lie superalgebras fall into the following classes:
\begin{itemize}
\item (classical type with nondegenerate Killing form) $A(m, n)$ where $m \neq n$, $B(m, n)$, $C(n)$, $D(m, n)$ where $m \neq n-1$, $F(4)$ or $G(3)$.


\item (classical type with vanishing Killing form) $A(n, n)$, $D(n+1, n)$, $D(2, 1; \alpha)$, $P(n)$ and $Q(n)$.


\item (non classical type) the Lie superalgebras of Cartan type $W(n)$, $S(n)$, $H(n)$ and $\widetilde{S}(n)$.
\end{itemize}
\end{prop}
Here
\begin{align*}
A(m, n) = \begin{dcases}
\slfrak(m+1 \mid n+1) & \text{for $m, n \geq 0$ and $m \neq n$} \\
\slfrak(m+1 \mid m+1) / \C I & \text{for $m = n \geq 1$} \\
\end{dcases}
\end{align*}
where
\begin{align*}
\slfrak(m \mid n) = \{X \in \glfrak(\C^{m|n}) \mid \str(X) = 0\}.
\end{align*}

The Lie superalgebra $\slfrak(m \mid n)$ contains subalgebras isomorphic to $\slfrak(m)$ and $\slfrak(n)$, namely the subalgebras of endomorphisms of the even and odd parts of $\C^{m|n}$. If $(\cdot, \cdot)$ denotes the supertrace form of the defining representation $\C^{m \mid n}$, and $\kappa(\cdot, \cdot)$ the Killing form then $\kappa(\cdot, \cdot) = 2(m-n) (\cdot, \cdot)$. This form is positive definite in $\slfrak(m)$ and negative definite in $\slfrak(n)$.

The families of simple Lie superalgebras $B(m, n)$, $C(n)$ and $D(m, n)$ are all particular cases of the orthosymplectic Lie superalgebras, whose definition we now recall. Consider $U = \C^{M|2n}$, equipped with $(\cdot, \cdot)$ a nondegenerate consistent supersymmetric bilinear form (i.e., symmetric on $U_{\even}$ and skew-symmetric on $U_{\odd}$). Let
\begin{align*}
\ospfrak(M \mid 2n)_{\even} &= \{ X \in \glfrak(U)_{\even} \mid (Xu, v) = -(u, Xv) \}. \\
\ospfrak(M \mid 2n)_{\odd} &= \{ X \in \glfrak(U)_{\odd} \mid (Xu, v) = -(-1)^{p(u)} (u, Xv) \}.
\end{align*}
Then by definition
\begin{align*}
B(m, n) &= \ospfrak(2m+1 \mid 2n), \qquad m \geq 0, n \geq 1 \\
C(n) &= \ospfrak(2 \mid 2n-2), \qquad n \geq 2, \\
\text{and} \quad D(m, n) &= \ospfrak(2m \mid 2n), \qquad m \geq 2, n > 0.
\end{align*}
Also $\ospfrak(0 \mid 2n) = \spfrak(2n)$ and $\ospfrak(M \mid 0) = \sofrak(M)$, and in general
\[
\ospfrak(M \mid 2n)_{\even} \cong \sofrak(M) \oplus \spfrak(2n)
\]
where $\sofrak(M)$ has its usual meaning for $M \geq 3$, while $\sofrak(0) = \sofrak(1) = 0$ and $\sofrak(2) = \C$. In the same way as for $\slfrak(m \mid n)$, the restrictions of the Killing form of $\ospfrak(M \mid 2n)$ are seen to be positive definite on $\sofrak(M)$ and negative definite on $\spfrak(2n)$.

Finally $F(4)_{\even} = A_1 \oplus B_3$ and $G(3)_{\even} = A_1 \oplus G_2$, and in both these cases the restrictions of the Killing form to the two components of the even part are respectively positive and negative definite.

\begin{defn}
A simple Lie superalgebra $\g = \g_{\even} \oplus \g_{\odd}$ is said to be \emph{basic} if $\g_{\even}$ is reductive, and if $\g$ possesses a nondegenerate even invariant bilinear form.
\end{defn}
The basic Lie superalgebras consist of the classical Lie superalgebras except $P(n)$ and $Q(n)$.
\begin{defn}[{\cite[Section 2]{KW1994}}]
The \emph{defect} of a basic simple Lie superalgebra $\g$ with root system $\Delta$, and in particular of a simple superalgebra with nondegenerate Killing form, is the dimension of a maximal isotropic subspace of $\R \Delta$.
\end{defn}
From the remarks above it follows that the only simple Lie superalgebras with defect equal to $0$ are the purely even ones and $B(0, n)$. Furthermore:
\begin{lemma}\label{lem:defect.root.norm}
Let $\g$ be a simple Lie superalgebra with nondegenerate Killing form, different than $B(0, n)$ and not purely even. Then $\g$ possesses an even root of positive norm and an even root of negative norm.
\end{lemma}

\begin{proof}
The relative norms of simple roots are listed in {\cite[Table 6.1]{KW2001}}.
\end{proof}

\section{Affine Lie superalgebras}\label{sec:affine.Lie}

Let $\g$ be a finite dimensional Lie superalgebra and $(\cdot, \cdot) : \g \times \g \rightarrow \C$ an even invariant bilinear form. We may form the corresponding (untwisted) affine Lie superalgebra
\begin{align*}
\widehat{\g} = \g[t, t^{-1}] \oplus \C K,
\end{align*}
\begin{align*}
[a t^m, bt^n] = [a, b]t^{m+n} + m \delta_{m, -n} (a, b) K, \qquad [K, \widehat{\g}] = 0.
\end{align*}
The universal vacuum module, of level $k \in \C$, is by definition the $\widehat{\g}$-module
\[
V^k(\g) = U(\widehat{\g}) \otimes_{U(\g[t] \oplus \C K)} \C\vac
\]
induced from the one-dimensional module $\C \vac$ in which $K \vac = k \vac$ and $\g[t] \vac = 0$.

Now we suppose that $\g$ is simple with nondegenerate Killing form $\kappa$ and root space decomposition $\g = \h \oplus \bigoplus_{\alpha \in \Delta} \g_\alpha$. The form $(\cdot, \cdot)$ induces a nondegenerate form on $\h^*$ which we also denote $(\cdot, \cdot)$. Simple calculation yields the following classical fact.
\begin{lemma}\label{lem:sing.vec}
Let $\g_{-\alpha} = \C f$ for some even root $\alpha$, and suppose that $f_{-1}^m \vac$ is a singular vector in $V^k(\g)$. Then
\begin{align*}
k = (m-1) \tfrac{(\alpha, \alpha)}{2}.
\end{align*}

\end{lemma}

\begin{proof}
Let $e \in \g_\alpha$ and put $h = [e, f] \in \h$. Since $\kappa$ is nondegenerate, we may write $\nu : \h \rightarrow \h^*$ for the induced map, and we have
\[
[e, f] = \kappa(e, f) \nu^{-1}(\alpha).
\]
By induction
\begin{align*}
e_1 f_{-1}^m \vac = (m (e, f) K - \tfrac{m(m-1)}{2} \alpha(h)) f_{-1}^{m-1} \vac.
\end{align*}

Since $\kappa(x, y) = \epsilon \cdot (x, y)$ for some constant $\epsilon$, if we write $\rho : \h \rightarrow \h^*$ for the map $x \mapsto (x, -)$ then $\nu = \epsilon \rho$. We have $(\alpha, \beta) = (\rho^{-1}(\alpha), \rho^{-1}(\beta))$. The singular vector condition above immediately becomes
\begin{align*}
K - (m-1) \tfrac{(\alpha, \alpha)}{2} = 0.
\end{align*}


\end{proof}

\section{Vertex superalgebras}

In this section we summarize definitions and basic results pertaining to vertex superalgebras.
\begin{defn}\label{def.vsa}
A vertex superalgebra consists of a vector superspace $V$, a distinguished vector $\vac \in V_{\even}$ called the vacuum vector, an even linear endomorphism $T \in \en(V)$ and, for every $a \in V$, a quantum field
\[
Y(a, z) = \sum_{n \in \Z} a_{(n)} z^{-n-1},
\]
i.e., a series $Y(a, z) \in V[[z, z^{-1}]]$ such that for all $b \in V$ one has $Y(a, z)b \in V((z))$. These data are to satisfy 
\begin{itemize}
\item $Y(\vac, z) = \text{id}_V$, also $Y(a, z) \vac \in V[[z]]$ and $Y(a, z)\vac|_{z=0} = a$,

\item $T \vac = 0$ and $[T, Y(a, z)] = \partial_z Y(a, z)$,

\item For all $a, b \in V$, and $n \in \Z$,
\[
Y(a, z) Y(b, w) i_{z, w} (z-w)^n - p(a, b) Y(b, w) Y(a, z) i_{w, z} (z-w)^n = \sum_{j \in \Z_+} Y(a_{(n+j)}b, w) \frac{\partial_w^j \delta(z, w)}{j!}.
\]
\end{itemize}
\end{defn}
Here $\delta(z, w)$ denotes the formal series $\sum_{n \in \Z} w^n z^{-n-1}$, and $i_{z, w}$ resp. $i_{w, z}$ denotes the operation of series expansion in positive powers of $w$, resp. $z$. Two useful consequences of the axioms are the skew symmetry formula
\[
Y(b, z)a = p(a, b) e^{zT} Y(a, -z)b,
\]
and the commutator formula
\[
[Y(a, z), Y(b, w)] = \sum_{j \in \Z_+} Y(a_{(j)}b, w) \frac{\partial_w^{j} \delta(z, w)}{j!}.
\]

\begin{defn}
A conformal structure on the vertex superalgebra $V$ of central charge $c$ consists of a vector $\omega \in V$ whose modes $L_n = \omega_{(n+1)}$ satisfy the relations of the Virasoro algebra at central charge $c$, i.e.,
\[
[L_m, L_n] = (m-n) L_{m+n} + \delta_{m, -n} \frac{m^3-m}{12} c \,\, \text{id}_V,
\]
the action of $L_0$ on $V$ is semisimple with finite dimensional eigenspaces (we write $V_n$ for the eigenspace with eigenvalue $n$), and $L_{-1} = T$. 
\end{defn}

\begin{defn}
The conformal vertex superalgebra $V$ is said to be of conformal field theory type (CFT type) if $V_n = 0$ for $n < 0$ and $V_0 = \C \vac$.
\end{defn}
The most commonly considered type of conformal vertex superalgebra is that for which $V = \bigoplus_{n \in (1/2)\Z_+} V_n$ with $V_n \subset V_{\even}$ for $n \in \Z$ and $V_n \subset V_{\odd}$ for $n \in 1/2+\Z$. In this article we instead focus attention on $\Z$-graded vertex superalgebras $V = \bigoplus_{n \in \Z_+} V_n$ where $V_n = V_{n, \even} \oplus V_{n, \odd}$ for each $n \in \Z_+$.

The following lemma is well known.
\begin{lemma}
Let $V$ be a conformal vertex superalgebra ($\Z$-graded or not). The restriction of the product $a \otimes b \mapsto a_{(0)}b$ to $V_1 \subset V$ descends to the quotient $V_1 / TV_0$, and equips the latter with the structure of a Lie superalgebra. In particular if $V$ is of CFT type then $V_1$ is a Lie superalgebra.
\end{lemma}

Let $V$ be a vertex superalgebra. A weak $V$-module is a vector superspace equipped with an operation $Y^M(a, z) \in \en{M}[[z, z^{-1}]]$ satisfying fairly evident analogs of the formulas in Definition \ref{def.vsa}. Admissible $V$-modules are weak $V$-modules satisfying additional grading and finiteness restrictions, we refer the reader to \cite{DLM-rational-regular} for the definition. The vertex superalgebra $V$ is said to be \emph{rational} if any admissible $V$-module is isomorphic to a direct sum of irreducible admissible $V$-modules.

The following notion of $C_2$-cofiniteness is an important technical hypothesis in Zhu's modularity theorem, to be be discussed below. 
\begin{defn}
The vertex superalgebra $V$ is said to be $C_2$-cofinite if the subspace $V_{(-2)}V$, spanned by elements of the form $a_{(-2)}b$, has finite codimension in $V$. The quotient $R_V = V / V_{(-2)}V$, equipped with $ab = a_{(-1)}b$ and $\{a, b\} = a_{(0)}b$ is a Poisson superalgebra. 
\end{defn}

The notion of invariant bilinear form on a $\Z$-graded vertex algebra goes back to \cite{B86} and such forms were studied in detail by Li in \cite{Li94}. We now review these results, or rather their natural extension to the context of $\Z$-graded vertex superalgebras.
\begin{defn}
Let $V$ be a $\Z$-graded conformal vertex superalgebra. A bilinear form $(\cdot, \cdot)$ on $V$ is said to be invariant if
\begin{align}\label{eq.def:invar.form.VA}
(Y(a, z)u, v) = p(a, u) (u, Y(e^{z L_1}(-z^{-2})^{L_0} a, z^{-1}) v)
\end{align}
for all homogeneous $a, u, v \in V$.
\end{defn}
If $V_0 = \C \vac$ then an example of an invariant bilinear form on $V$ is $\left<\cdot, \cdot\right>$ defined by
\begin{align}\label{eq:std.form}
\left<u, v\right> \vac = -\res_z z^{-1} Y(e^{z L_1}(-z^{-2})^{L_0} u, z^{-1}) v.
\end{align}
If we assume that $L_1 V_1 = 0$ then for $u, v \in V_1$ we have
\[
\left<u, v\right> = u_{(1)}v.
\]
\begin{prop}
Let $V$ be a $\Z$-graded conformal vertex superalgebra.
\begin{itemize}
\item An invariant bilinear form $(\cdot, \cdot)$ on $V$ is supersymmetric, i.e.,
\[
(u, v) = p(u, v) (v, u).
\]

\item Let $M$ be a $V$-module, and $M^{\text{v}} \subset M$ its set of \emph{vacuum-like vectors}, i.e.,
\begin{align*}
M^{\text{v}} &= \{m \in M \mid \text{$a_{(n)}m = 0$ for all $a \in V$ and $n \in \Z_{\geq 0}$} \} \\
&= \{m \in M \mid L_{-1} m = 0\}.
\end{align*}
The space $\Hom_V(V, M)$ of morphisms of $V$-modules from $V$ to $M$ is isomorphic to $M^{\text{v}}$ under the assignment $\psi \mapsto \psi(\vac)$.

\item The space of invariant bilinear forms on $V$ is isomorphic to $(V_0 / L_1 V_1)^*$.
\end{itemize}
\end{prop}

\begin{proof}
The first part is {\cite[Proposition 5.3.6]{FHL}} for the case of $\Z$-graded vertex algebras. The same calculation, adapted to the supersymmetric case via the Koszul rule of signs, yields the claim. The second and third claims are similarly obtained by adapting, respectively, Proposition 3.4 and Theorem 3.1 of \cite{Li94} to the supersymmetric case.

%
\end{proof}

\begin{cor}\label{cor:simple.form.nondeg}
Let $V$ be a $\Z$-graded conformal vertex superalgebra. If $V$ is simple, of CFT type, and $L_1 V_1 = 0$, then the bilinear form \eqref{eq:std.form} is nondegenerate.
\end{cor}

A vertex superalgebra $V$ is said to be self-contragredient if it possesses a nondegenerate invariant bilinear form.

\begin{rem}
The rendition of the notion of invariant bilinear form to the case of $\tfrac{1}{2}\Z$-graded vertex superalgebras is more complicated than it is in the $\Z$-graded case, essentially because of the appearance of $(-z^{-2})^{L_0}$. See {\cite[Equation (2.6.1)]{Duncan.supermoonshine}} and {\cite[Section 2.2]{Yamauchi}}.
\end{rem}

\begin{defn}\label{def:holom}
Let $V$ be a $\Z$-graded conformal vertex superalgebra. We shall call $V$ \emph{holomorphic} if it is self-contragredient, $C_2$-cofinite and rational and if, moreover, the unique irreducible ordinary $V$-module is the adjoint module $V$ itself (in particular $V$ is simple).
\end{defn}

\section{Affine vertex superalgebras}

Let $\widehat{\g}$ and $V^k(\g)$ be as in Section \ref{sec:affine.Lie} above. This $\widehat{\g}$-module can be given the structure of a $\Z$-graded vertex superalgebra of CFT type in which $V_1$ is naturally identified with $\g$ and for $x \in V_1$ we have $Y(x, z) = \sum_{n \in \Z} (x t^n) z^{-n-1}$. The simple quotient of $V^k(\g)$ is denoted $V_k(\g)$.

If $(\cdot, \cdot)$ is non degenerate then the Sugawara construction equips $V^k(\g)$ with a conformal structure as follows {\cite[p. 165]{Kac.book}}. Let $\{x^i\}$ and $\{y^i\}$ be a pair of dual bases of $\g$ with respect to $(\cdot, \cdot)$ and let
\[
\omega = \frac{1}{2(k+h^\vee)} \sum_{i}  x^{i}_{(-1)} y^{i}_{(-1)} \vac.
\]
Then $\omega$ is a conformal vector of central charge
\[
c = \frac{k \sdim(\g)}{k + h^\vee}.
\]
Here $h^\vee$ is the dual Coxeter number of the pair $\g, (\cdot, \cdot)$, defined as half the eigenvalue of the Casimir operator
\[
\Omega = \sum_i x^i y^i \in U(\g)
\]
in the adjoint representation. If $\g$ is simple then there is at most one invariant nondegenerate form, and it is customary to fix it so that $(\theta, \theta)=2$ where $\theta$ is the highest root of $\g$. The corresponding dual Coxeter number $h^\vee$ is listed for some cases in Table \ref{table:ddcn} below.

We note that $L_1 V^k(\g)_1 = 0$, so \eqref{eq:std.form} equips this vertex superalgebra, as well as the simple quotient $V_k(\g)$, with symmetric invariant bilinear forms $\left<\cdot, \cdot\right>$. The restriction of  
$\left<\cdot, \cdot\right>$ to $V_1 = \g$ is just $k$ times $(\cdot, \cdot)$.

\section{The theorem of Zhu and its generalizations}

In \cite{Zhu} Zhu proved a remarkable theorem on modularity of characters of vertex algebras. This theorem has been generalized in a number of directions. Most relevant for us are \cite{Dong-Zhao-Z} and \cite{JVE-CMP} in which $\Z$-graded vertex superalgebras are treated.

Before proceding we recall some notational conventions on modular forms. The Eisenstein series
\begin{align*}
{G}_{2k}(q) &= 2\zeta(2k) + \frac{2(2\pi i)^{2k}}{(2k-1)!} \sum_{n=1}^\infty \sigma_{2k-1}(n) q^n, \\
\text{and} \quad E_{2k}(q) &= 1 - \frac{4k}{B_{2k}} \sum_{n=1}^\infty \sigma_{2k-1}(n) q^n,
\end{align*}
are related by $G_{2k}(q) / E_{2k}(q) = 2\zeta(2k) = -(2\pi i)^{2k} B_{2k} / (2k)!$. Here $\zeta(s)$ is the Riemann zeta function, $B_{2k}$ are the Bernoulli numbers (defined by $x/(e^x-1) = \sum_{n=0}^\infty B_n x^n / n!$) and $\sigma_\ell(n) = \sum_{d \mid n} d^\ell$ is a divisor power sum. 
For $2k \geq 4$ the Eisenstein series is a holomorphic modular form on $\Gamma = SL_2(\Z)$ of weight $2k$. The dimension of the space $\mathcal{M}_{2k}^{\text{hol}}(\Gamma)$ of such forms is immediately deducible from the fact that $\mathcal{M}_{*}^{\text{hol}}(\Gamma) = \bigoplus_{k \in \Z_{+}} \mathcal{M}_{2k}^{\text{hol}}(\Gamma) \cong \C[G_4, G_6]$ as graded rings. In particular $\mathcal{M}_{12}^{\text{hol}}(\Gamma)$ is $2$-dimensional, spanned by $E_{12}(q)$ and the cusp form
\begin{align*}
\Delta(q) = q - 24 q^2 + 252 q^3 + \cdots,
\end{align*}
and $\mathcal{M}_{14}^{\text{hol}}(\Gamma)$ is $1$-dimensional, spanned by
\begin{align*}
E_{14}(q) = 1 - 24 q - 196632 q^2 - \cdots.
\end{align*}
Unlike $E_{2k}(q)$ for $k \geq 2$, the Eisenstein series
\begin{align*}
E_2(q) = 1 - 24 q - 72 q^2 - \cdots
\end{align*}
is only quasi-modular.

Let $V$ be a $\Z$-graded conformal vertex superalgebra of central charge $c$. First a modified vertex superalgebra structure $(V, Y[-, z], \vac', \omega')$ is defined as follows. For homogeneous $a \in V$ one has
\[
Y[a, z] = \sum_{n \in \Z} a_{([n])} z^{-n-1} = e^{2\pi i \Delta(a)} Y(a, e^{2\pi i z}-1),
\]
also $\vac' = \vac$ and $\omega' = (2\pi i)^2 (\omega - \tfrac{c}{24}\vac)$. In particular $L_{[0]} = \omega_{([1])}'$ defines a new $\Z_+$-grading on $V$. We write
\[
V_{[k]} = \{v \in V \mid L_{[0]}v = k v\}.
\]

Now let $M$ be an ordinary $V$-module. Let $u \in V$ and $q = e^{2\pi i \tau}$ where $\tau$ lies in the upper half complex plane, equivalently $|q| < 1$. The supertrace function is defined to be
\[
Z_M(u, \tau) = \str_M u_0 q^{L_0-c/24}.
\]
This is a formal power series. If $V$ is $C_2$-cofinite then the series converges to a holomorphic function (for each $u \in V$). The supertrace function at $u = \vac$ recovers the supercharacter
\[
\chi_M(\tau) = Z_M(\vac, \tau) = q^{h-c/24} \sum_{n \in \Z_{\geq 0}} \sdim(M_n) q^n,
\]
where $M = \bigoplus_{n \in \Z_{\geq 0}} M_n$ with $L_0|_{M_n} = n + h$. 
Now Zhu's theorem and its generalizations (see {\cite[Theorem 5.3.2]{Zhu}} for vertex algebras, {\cite[Theorem 4.5]{Dong-Zhao-Z}} and {\cite[Theorem 1.3]{JVE-CMP}} for $\Z$-graded vertex superalgebras) assert modular invariance of the supertrace functions $Z_M$ in an appropriate vector-valued sense. As a very special case we have:
\begin{thm}\label{thm:zhu.main}
Let $V$ be a $\Z$-graded holomorphic vertex superalgebra of central charge $c=24$. Then
\[
Z_{V}\left( u, \frac{a\tau+b}{c\tau+d} \right) = (c\tau+d)^k Z_{V}(u, \tau)
\]
for all $k \in \Z_+$ and $u \in V_{[k]}$, and all $A = \twobytwo abcd \in SL_2(\Z)$.
\end{thm}

\begin{proof}
The theorem of Zhu asserts that
\[
Z_{V}\left( u, \frac{a\tau+b}{c\tau+d} \right) = (c\tau+d)^k \rho(A) Z_{V}(u, \tau)
\]
for some character $\rho : SL_2(\Z) \rightarrow \C^\times$. Since $V$ is $\Z$-graded and $c = 24$, the action of $T : \tau \mapsto \tau + 1$ on $Z_V(u, \tau)$ is trivial. The action of $S : \tau \mapsto -1/\tau$ is also trivial since in general $S^2$ and $(ST)^3$ act trivially. 
\end{proof}

\begin{rem}
Modularity of supertrace functions holds, in a vector-valued sense, for rational $C_2$-cofinite vertex superalgebras. We refrain from spelling out the details since, in the supersymmetric case, the full statement requires \emph{queer supertrace functions}, which we do not need in this work. See \cite{JVE-CMP} for details.
\end{rem}

We will use the following intermediate result of Zhu ({\cite[Proposition 4.3.5]{Zhu}} for vertex algebras, {\cite[Proposition 8.4]{JVE-CMP}} for $\Z$-graded vertex superalgebras).
\begin{prop}\label{prop:zhu.prelim}
Let $V$ be a rational and $C_2$-cofinite $\Z$-graded vertex superalgebra, and let $M$ be a positive energy $V$-module. For every $u, v \in V$ we have
\begin{align*}
\str_M u_0 v_0 q^{L_0 - c/24} = Z_{M}(u_{([-1])}v, \tau) - \sum_{k=1}^\infty {G}_{2k}(q) Z_M(u_{([2k-1])}v, \tau).
\end{align*}
\end{prop}

One readily checks that if $u \in V_1$ and $L_1u = 0$ then $L_{[0]}u = L_0 u = u$. Note also that $u{([1])}v = (2\pi i)^2 u_{(1)}v = (2\pi i)^2 \left<u, v\right> \vac$, and $G_2(q) = -(2\pi i)^2 E_2(q) / 12$.
\begin{cor}\label{cor:Zhu.it}
Suppose $V$ to be holomorphic. For all $u, v \in V_1$ we have
\begin{align}\label{eq:Zhu.it}
\str_{V} u_0 v_0 q^{L_0-c/24} = Z_{V}(u_{([-1])}v, \tau) + \frac{\left<u, v\right>}{12} {E}_{2}(q) \chi_V(\tau).
\end{align}
\end{cor}

\section{Holomorphic vertex superalgebras of central charge $24$}

The following proposition is similar to {\cite[Corollary 2.3]{DM04}} in the purely even case.
\begin{prop}\label{prop:V1.Killing.nondeg}
Let $V$ be a holomorphic $\Z$-graded vertex superalgebra of central charge $c=24$ and let $\g$ denote the Lie superalgebra $V_1$. Then either the Killing form of $\g$ is nondegenerate, or else $\sdim(\g) = 24$ and the Killing form of $\g$ vanishes.
\end{prop}

\begin{proof}
Let us fix $u, v \in V_1$ and write
\[
\sum_{n=0}^\infty a_n q^n = \Delta(q) \str_V (u_{([-1])}v)_0 q^{L_0-c/24}.
\]
Since $u_{([-1])}v$ has conformal weight $2$ (relative to Zhu's modified vertex algebra structure) we see that $\sum a_n q^n$ is a holomorphic modular form of weight $14$. It is therefore a multiple of $E_{14}(q) = 1 - 24q + \cdots$, so $a_1 = -24a_0$.

Note that $u_0|_{V_0} = 0$ and
\[
\str_{V_1} u_0 v_0 = \kappa(u, v).
\]
Multiplying equation \eqref{eq:Zhu.it} through by $\Delta(q)$ and equating coefficients yields
\begin{align*}
(a_0 + a_1 q + \cdots) = \left(0 + \kappa(u, v) q + \cdots \right) - \frac{\left<u, v\right>}{12} \left(1 - 48 q + \cdots\right) \left(1 + \sdim(V_1) + \cdots\right).
\end{align*}
We immediately deduce
\begin{align}\label{eq:fundamental.relation}
\kappa(u, v) = \frac{\sdim(V_1)-24}{12} \left<u, v\right>.
\end{align}
Since $V$ is self-contragredient, Corollary \ref{cor:simple.form.nondeg} assures us that $\left<\cdot, \cdot\right>$ is nondegenerate. The proposition follows immediately.
\end{proof}

We recall a basis theorem for $C_2$-cofinite vertex algebras, proved in \cite{GN} (see also \cite[Theorem 3.1]{ABD}), and which also holds for vertex superalgebras with the same proof (see \cite{Dong.Han.finite}).
\begin{thm}\label{thm:PBW.basis}
Let $V$ be a $C_2$-cofinite vertex superalgebra of CFT type. Let $\{a^1, a^2, \ldots, a^r\}$ be a set of elements of $V$, homogeneous (with respect to the $\Z_+$-grading and the $\Z/2$-grading) of positive conformal weight which, together with $\vac$, span $R_V = V / V_{(-2)}V$. Then $V$ is spanned by the monomials
\begin{align*}
a^{i_1}_{(-n_1-1)} a^{i_2}_{(-n_2-1)} \cdots a^{i_k}_{(-n_k-1)} \vac
\end{align*}
where $n_1 > n_2 > \ldots > n_k \geq 0$.
\end{thm}

\begin{prop}\label{prop:C2.implies.integ}
Let $V$ be a holomorphic $\Z$-graded vertex superalgebra of central charge $c=24$ for which $\sdim(V_1) \neq 24$. Then the defect of $V_1$ is $0$.
\end{prop}

\begin{proof}
By Proposition \ref{prop:V1.Killing.nondeg} the Killing form of $\g = V_1$ is nondegenerate, and by Proposition \ref{prop.nondeg=good}, $\g$ is a direct sum of simple components with the nondegenerate Killing form. If the defect of $\g$ were positive, the same would be true of one of its simple components and, by Lemma \ref{lem:defect.root.norm} there would exist even roots $\alpha_1$ and $\alpha_2$ of $\g$ for which $\kappa(\alpha_1, \alpha_1) > 0$ and $\kappa(\alpha_2, \alpha_2) < 0$.   

We now proceed as in the proof of {\cite[Theorem 3.1]{DM06}}. Let $\alpha$ be an even root of $\g$ and $H^{\alpha}$ the corresponding coroot, let $\g_\alpha = \C e$, and write $V_\Delta^{(t)} = \{v \in V_\Delta \mid H^\alpha_{(0)} v = t v \}$. Without loss of generality we may take the elements $a^i$ of Theorem \ref{thm:PBW.basis} to be homogeneous relative to $\alpha$. By Theorem \ref{thm:PBW.basis} there exists a positive constant $C$ such that 
\[
\dim(V_{\Delta}^{(t)}) = 0, \qquad \text{unless $\Delta > C t^2$}.
\]
It follows that $e_{(-1)}^{m+1} \vac = 0$ in $V$ for some $m \in \Z_{\geq 0}$. 

The vertex subalgebra of $V$ generated by $\g = V_1$ is a quotient of $V^k(\g)$ at some level $k$ (and choice of normalization of the form $(\cdot, \cdot)$), and $e_{-1}^{m+1} \vac$ must be a singular vector in $V^k(\g)$. It follows from Lemma \ref{lem:sing.vec} that $k \cdot (\alpha, \alpha) > 0$. Since this cannot hold for both $\alpha_1$ and $\alpha_2$ we have a contradiction, and so the defect of $\g$ is $0$.
\end{proof}

\begin{thm}\label{thm:main.24}
Let $V$ be a holomorphic $\Z$-graded vertex superalgebra of central charge $24$ for which $\sdim(V_1) \neq 24$. Then $V_1$ is either zero, or is isomorphic to a finite direct sum
\[
\g^{(1)} \oplus \cdots \oplus \g^{(r)}
\]
of simple Lie superalgebras, each even or of type $B(0, n)$. The restriction of the form $\left<\cdot, \cdot\right>$ of \eqref{eq.def:invar.form.VA} to $\g^{(i)}$ is some positive integer $k_i$ times its normalized invariant form $(\cdot, \cdot)$. Furthermore
\begin{align}\label{eq:key.rel.var}
\frac{h_i^\vee}{k_i} = \frac{\sdim(V_1)-24}{24}, \qquad i = 1, 2, \ldots, r.
\end{align}
The number of such Lie superalgebras is exactly 1332.
\end{thm}

\begin{proof}
By Proposition \ref{prop:V1.Killing.nondeg} the Killing form of $V_1$ is nondegenerate, hence by Proposition \ref{prop.nondeg=good} it decomposes as a direct sum of simple Lie superalgebras. By Proposition \ref{prop:C2.implies.integ}, each of these simple components has defect $0$. Therefore each simple component is either purely even or else isomorphic to $B(0, n)$ for some $n \geq 1$.

Let $V_1 = \g^{(1)} \oplus \cdots \oplus \g^{(r)}$ be the decomposition into simple components, and let $(\cdot, \cdot)_i$ be the nondegenerate invariant bilinear form on $\g^{(i)}$ normalised so that long roots have norm $2$. Then by Lemma \ref{lem:sing.vec} the restriction of $\left<\cdot, \cdot\right>$ to $\g^{(i)}$ is some multiple $k_i$ of $(\cdot, \cdot)_i$. That $k_i$ is a positive integer follows from the proof of Proposition \ref{prop:C2.implies.integ}.

Now it well known that $\kappa(\cdot, \cdot) = 2h^\vee_i (\cdot, \cdot)_i$, where $h^\vee$ is the dual Coxeter number given in Table \ref{table:ddcn} below. So \eqref{eq:fundamental.relation} turns into \eqref{eq:key.rel.var}. By inspection of Table \ref{table:ddcn} below, we see that $\sdim / (h^\vee)^2 \geq 4/9$ (with equality realized at $B(0, 1)$). For each simple component $\g$ of $V_1$ we have
\begin{align*}
h^\vee \geq \frac{h^\vee}{k} = \frac{\sdim(V_1)}{24} - 1 \geq \sdim(\g)/24 - 1 \geq (h^\vee)^2 / 54 - 1. 
\end{align*}
This implies $h^\vee < 55$ and in turn $\sdim(V_1) < 1344$.

Using a computer search, written in \texttt{Python}, all direct sums of simple Lie superalgebras from Table \ref{table:ddcn} with total superdimension at most $1344$, satisfying \eqref{eq:key.rel.var} are determined. Such algebras are finite in number since all allowable simple components have positive superdimension. We find 1332 solutions.
\end{proof}

\begin{rem}
The simple affine vertex superalgebra associated with $B(0 \mid n)$ at integer level is known to be $C_2$-cofinite and rational \cite{KG2011} \cite{CL}, and so the same is true of the affine vertex superalgebras $\left<V_1\right>$ appearing in Theorem \ref{thm:main.24}.
\end{rem}

\begin{rem}
The inclusion of the affine subalgebra $\left<V_1\right>$ into $V$ presents the latter as a finite extension of the former. Techniques of \cite{HKL}, and \cite{McRae} can then be applied. Already in the purely even case, however, the decomposition of $V$ into modules over $\left<V_1\right>$ can be very complicated. One solution, adopted in \cite{EMS20}, \cite{Lam.Shimakura.framed}, etc. (see \cite{vE.survey} for a survey), is to build $V$ as an extension of $W^G$ for some holomorphic $W$, and $G$ a cyclic group. The extension theory of $W^G$ is much easier to handle than that of $\left<V_1\right>$, but finding appropriate $W$ and $G$ is difficult.
\end{rem}

\begin{table}[h]
\caption{Dimensions and dual Coxeter numbers} \label{table:ddcn} \centering%
\begin{tabular}{|c|c|c|c|}
\hline
$\g$ & $\dim(\g)$ & $h^\vee$ \\
\hline
$A_n$ & $n(n+2)$ & $n+1$ \\
$B_n$ & $n(2n+1)$ & $2n-1$ \\
$C_n$ & $n(2n+1)$ & $n+1$ \\
$D_n$ & $n(2n-1)$ & $2n-2$ \\
$E_6$ & $78$ & $12$ \\
$E_7$ & $133$ & $18$ \\
$E_8$ & $248$ & $30$ \\
$F_4$ & $52$ & $9$ \\
$G_2$ & $14$ & $4$ \\
$B(0, n)$ & $n(2n+1)\mid 2n$  & $n + 1/2$ \\
\hline
\end{tabular}
\end{table}

\section{Holomorphic vertex superalgebras of central charge $8$ and $16$}

Throughout this section $V$ will be a holomorphic $\Z$-graded vertex superalgebra of central charge $c = 8$ or $c = 16$.

As in the proof of Theorem \ref{thm:zhu.main}, we consider the supercharacter
\begin{align*}
\chi_V(\tau) 
&= Z_V(\vac, \tau) = \str_V q^{L_0 - c/24} \\
&= q^{-c/24} \left( 1 + \sdim(V_1)q + \sdim(V_2) q^2 + \cdots \right),
\end{align*}
where $q = e^{2\pi i \tau}$. The supercharacter now satisfies
\[
\chi_V\left(A \cdot \tau\right) = \rho(A) \chi_V(\tau)
\]
for some character $\rho : SL_2(\Z) \rightarrow \C^\times$. In fact $\rho(T) = e^{2\pi i / 3}$ if $c=8$ and $\rho(T) = e^{-2\pi i / 3}$ if $c=16$, so $\rho(S) = 1$ since $S^2$ and $(ST)^3$ both act trivially.

The kernel in $SL_2(\Z)$ of $\rho$ is a congruence subgroup of level $3$ sometimes denoted $\Gamma^3$ \cite{Newman} \cite{Stein.modular.book}. It is a genus zero group, and the cube root of the $j$-invariant
\[
j(\tau)^{1/3} = q^{-1/3} \left( 1 + 248 q + \cdots \right)
\]
is a Hauptmodul.

Suppose $c=8$. Then $\chi_V(\tau) - j(\tau)^{1/3}$ is a modular function for $\Gamma^3$ holomorphic on the upper half plane (hence constant), and without constant term in its series expansion (hence zero in fact). It follows that $\sdim(V_1) = 248$.

The unique holomorphic vertex algebra of central charge $8$ is the lattice vertex algebra $V(E_8)$ \cite{DM04}. It is clear that $V \oplus V(E_8) \oplus V(E_8)$ is a holomorphic vertex superalgebra of central charge $24$. By Theorem \ref{thm:main.24} we have $V_1$ a sum of simple components, and
\begin{align*}
\frac{h^\vee}{k} = \frac{\sdim(V_1)+496-24}{24}
\end{align*}
for each simple component. A computer search, done using \texttt{Python}, yields the unique solution $V_1 = E_8$ at level $k=1$. We can conclude that $V = V(E_8)$, indeed the equality $V_1 = E_8$ at level $1$ induces an embedding $V(E_8) \subset V$, which yields a decomposition of the latter as a direct sum of irreducible modules of the former. Since $V(E_8)$ is holomorphic, we must have equality.

Now suppose $c=16$. Similar arguments as above show that $\sdim(V_1) = 496$ and, since $V \oplus V(E_8)$ is holomorphic of central charge $24$, the structure of $V_1$ is either $E_8 \oplus E_8$ or else $D_{16}$. Since these Lie algebras are pure even and are each the weight-one space of a unique holomorphic vertex algebra (the lattice vertex algebras $V(E_8) \oplus V(E_8)$ and $V(D_{16}^+)$) \cite{DM04}, the same reasoning as above implies that $V$ itself must be one of these two vertex algebras.
\begin{thm}\label{thm:main.8.16}
Every holomorphic $\Z$-graded vertex superalgebra of central charge $8$ or $16$, is purely even.
\end{thm}

\section{Conclusion and outlook}

In this note we have ruled out the existence of non purely even holomorphic $\Z$-graded vertex superalgebras of central charge $8$ and $16$, and in the case of central charge $24$ we have put constraints on the possible Lie superalgebra structures of the weight-one space. Evidently not all of the resulting 1332 Lie superalgebras are realized as weight-one spaces, indeed of the 221 purely even Lie superalgebras on the list, only 69 are realized.

It would be interesting to apply the techniques used in \cite{Sch93} and \cite{EMS20} to reduce the list of candidate weight-one spaces. It would also be highly desirable to exhibit examples of (non purely even) holomorphic vertex superalgebras, or else to find a conceptual explanation for their non existence. The two main methods of construction of examples in the even case are (1) the lattice vertex algebra construction applied to self-dual even lattices, and (2) the construction of new holomorphic vertex algebras from old as orbifolds. Neither technique adapts straightforwardly to the present context, as vertex superalgebras associated with odd lattices are $(1/2)\Z$-graded (unless the conformal structure is modified, in which case the self-contragredient property is lost), and the orbifold construction also tends to produce $(1/2)\Z$-graded vertex superalgebras, and not $\Z$-graded ones. On the other hand examples of $\Z$-graded vertex superalgebras which are rational, $C_2$-cofinite, self-contragredient, etc., exist, $V_k(\mathfrak{osp}(1|2))$ for example. So there seems to be no obvious obstruction to the existence of holomorphic examples.

The case of holomorphic $\Z$-graded vertex superalgebras $V$ of central charge $24$ satisfying $\sdim(V_1) = 24$, in which case the Killing form of $V_1$ necessarily vanishes, remains to be investigated.

The definition of holomorphic vertex algebra (cf. Definition \ref{def:holom} above) includes a number of technical conditions which experience has proven to be appropriate in the pure even case. Examples of vertex superalgebras, such as $V_1(\g)$ for $\g$ either $G(3)$ or $F(4)$ \cite{KW2001} \cite{Ai.Lin}, suggest however that the appropriate technical conditions in the supersymmetric case might be different. The appearance of mock modular forms as characters of affine superalgebras, discovered by Kac and Wakimoto \cite{KW2001} and now the subject of an extensive literature, is another indication that the supersymmetric case is substantially different than the even case. We hope to continue to investigate these matters in future work.

\bibliographystyle{plain}
\bibliography{refs}

\end{document}